\documentclass[12pt,a4paper]{amsart}

\title{Computing the Weil representation of a superelliptic curve}
\author{Irene I.~Bouw, Duc Khoi  Do, and Stefan Wewers}
\makeatletter
\usepackage[utf8]{inputenc}
\usepackage[english]{babel}
\usepackage{amsmath, amssymb, mathtools}

\usepackage[usenames,dvipsnames]{color}
\usepackage[all,cmtip]{xy}
\usepackage{fullpage}
\numberwithin{equation}{section}
\numberwithin{figure}{section}
\usepackage[all,cmtip]{xy}
\usepackage{longtable}
\usepackage{enumerate}

\newtheorem{thm}{Theorem}[section]
\newtheorem{cor}[thm]{Corollary}
\newtheorem{prop}[thm]{Proposition}
\newtheorem{lem}[thm]{Lemma}

\theoremstyle{definition}
\newtheorem{defn}[thm]{Definition}
\newtheorem{exa}[thm]{Example}

\newtheorem{assumption}[thm]{Assumption}
\newtheorem{rem}[thm]{Remark}

\newcommand{\NN}{\mathbb{N}}
\newcommand{\ZZ}{\mathbb{Z}}
\newcommand{\QQ}{\mathbb{Q}}
\newcommand{\FF}{\mathbb{F}}

\newcommand{\CC}{\mathbb{C}}
\newcommand{\PP}{\mathbb{P}}

\newcommand{\GL}{\mathop{\rm GL}\nolimits}

\newcommand{\Gal}{\mathop{\rm Gal}\nolimits}

\newcommand{\p}{\mathfrak{p}}
\newcommand{\OO}{\mathcal{O}}
\newcommand{\Spec}{\mathop{\rm Spec}}

\newcommand{\Aut}{\mathop{\rm Aut}\nolimits}

\newcommand{\Frob}{\mathop{\rm Frob}\nolimits}

\newcommand{\Char}{\mathop{\rm char}\nolimits}

\newcommand{\Y}{\mathcal Y}

\newcommand{\Yb}{\overline{Y}}
\newcommand{\Xb}{\overline{X}}

\newcommand{\nr}{{\scriptstyle \rm nr}}

\newcommand{\et}{{\scriptstyle \rm et}}

\newcommand{\Tr}{{\rm Tr}}
\newcommand{\abs}[1]{\lvert#1\rvert}
\newcommand{\Id}{{\rm Id}}

\newcommand{\Kb}{\overline{K}}

\usepackage{hyperref}
\hypersetup{
	colorlinks=true,
	linkcolor=blue,
	filecolor=magenta,      
	urlcolor=blue,
}
\begin{document}

\begin{abstract}
	We study the Weil representation $\rho$  of a curve over a $p$-adic field  with potential reduction of compact type. We show that $\rho$ can be reconstructed from its stable reduction.  For superelliptic curves of the form $y^n=f(x)$ at primes $\p$ whose residue characteristic is
	prime to the exponent $n$ we make this explicit.
	
	\noindent 2020 \emph{Mathematics Subject Classification}. Primary
	11F80. Secondary: 14H25, 11G20, 11S40.
\end{abstract}

\maketitle
\section*{Introduction}

Let $Y/K$ be a smooth, projective, and absolutely irreducible curve over a finite extension $K/\QQ_p$. For simplicity, we assume that $g:=g(Y)\geq 2$. We fix an algebraic closure $\overline{K}$ of $K$ and write $Y_{\overline{K}}=Y\otimes_K \overline{K}$. We choose an auxiliary  prime $\ell\neq p$, and write $V_\ell:=H^1_\et(Y_{\overline{K}}, \QQ_\ell)$. The absolute  Galois group $\Gamma_K$ of $K$ acts naturally  on $V_\ell$. This defines the $\ell$-adic representation
\[
\rho_\ell:\Gamma_K\to \GL_{\QQ_\ell}(V_\ell),
\]
which is the key object of study in this paper. 

In this paper we assume that the curve $Y$ has potential  reduction of compact type. This means that the Jacobian of $Y$ has good reduction, after replacing $K$ by a finite extension, if necessary. 
In this case the $\ell$-adic representation $\rho_\ell$ defines a Weil representation $\rho:W_K\to \GL_\CC(V)$, which is independent of the choice of the auxiliary prime $\ell$. It is a complex representation of the Weil group $W_K<\Gamma_K$ that determines $\rho_\ell$.  We refer to Theorem \ref{thm:compact_type} for the precise statement. 

Let $\Gamma_{\FF_K}$ be the absolute Galois group of the residue field $\FF_K$ of $K$. We write $Z<\Gamma_{\FF_K}$ for the subgroup generated by the Frobenius automorphism  $\Frob_{\FF_K}$. The Weil group fits into an exact sequence
\[
1\to I_K\to W_K\stackrel{r}{\to} Z\to 1.
\]
A \emph{Frobenius element} $\phi$ is an element of $W_K$ such that $r(\phi)=\Frob_{\FF_K}.$

Write $K^{\nr}$ for the maximal unramified subextension of $\overline{K}/K$. There exists a finite Galois extension $L/K^{\nr}$ such that $Y_L:=Y\otimes_K L$ has a semistable model $\mathcal{Y}/\mathcal{O}_L$. We write $\Yb$ for the normalization of the special fiber of $\mathcal{Y}$. The curve $\overline{Y}$ admits a model $\overline{Y}_0$ defined over $\FF_K$, see Lemma \ref{lem:phi-models}. The assumption that $Y$ has reduction of compact type over $L$ implies that
\[
H^1_{\et}(Y_{\overline{K}}, \QQ_\ell)\simeq H^1_{\et}(\overline{Y}, \QQ_\ell)=:\overline{V}_\ell,
\]
see Proposition \ref{prop:compact_type}. 
The action of $\Gamma_K$ on $V_\ell$ factors through a subgroup
$G\subset \Aut_{\FF_K}(\Yb)$. 

Our method to compute the Weil representation $\rho$  uses a result of Dokchitser--Dokchitser \cite{DD}, which states that $\rho$ is determined by the local polynomial 
\[
P(\phi^{-1}, T):=\det(1-\phi T|_{\overline{V}_\ell}),
	\]
for the different Frobenius elements $\phi$. Since the image of $I_K$ is finite, it suffices to compute finitely many local polynomials.  
In Lemma \ref{lem:phi-models} we prove that there is a correspondence between Frobenius elements $\phi$ in the Weil group, and twists $\Yb_\phi$ of the model $\Yb_0$.  We may therefore compute the local polynomial $P(\phi^{-1}, T)$ by point counting on the curve $\Yb_\phi$. Combining the information on the different  Frobenius elements allows us to compute the action of $G$ on $\overline{V}_\ell$ explicitly from a stable model of $Y$.  Our method can be extended to curves that do not have reduction of compact type. In this case, one needs to replace the Weil representation by a Weil--Deligne representation. We come back to this in a future paper.

From  \S\ref{sec:superell} we assume that $Y$ is a superelliptic curve, i.e.~a curve defined by an affine equation
\[
   Y:\; y^m=f(x), \qquad f(x)\in K[x].
\] 
We consider the case that $Y$ has potential reduction of compact type and that $p\nmid m$. This is a condition that is easy to check. Moreover, in our previous paper \cite{superell}, we described how to determine the stable reduction of $Y$. Building on these results, we show in  \S\ref{sec:supermodels}  that the correspondence from Lemma \ref{lem:phi-models} between Frobenius elements $\phi$ and models $\Yb_\phi$ can be made completely explicit. We therefore obtain a practical procedure for computing the Weil representation of a superelliptic curve with  potential reduction of compact type.  We apply our method in \S\ref{sec:exa} to two Picard curves with potential good reduction to characteristic $p=2$ and one with potential reduction of compact type.

\section{The local Galois representation of a curve}\label{sec:localrep}

In this section, we recall some properties of the $\ell$-adic representation of a curve defined over a  local field. We restrict to  the case of curves whose reduction is potentially of compact type. In this case the $\ell$-adic representation defines a Weil representation. For more details on Weil representations, we refer to \cite{BH}, Chapters 28, 32, \cite{Tate}, and \cite{Rohrlich}.

\subsection{Weil representations}

Let $p$ be  prime and  $K/\QQ_p$ a finite extension. We choose an algebraic closure $\overline{K}$ of $K$, and fix it in the rest of this paper. 
 We denote the normalized  valuation of $K$ by $v_K$, the ring of integers by $\mathcal{O}_K$, and the absolute Galois group of $K$ by $\Gamma_K=\Gal(\bar{K}/K)$.  We write $q=p^s$ for the cardinality of the residue field $\FF_K$ of $K$. We write $k$ for the residue field of the unique extension of $v_K$ to $\overline{K}$. Note that $k$ is an algebraic closure of $\FF_K$. We obtain a canonical morphism
\begin{equation} \label{eq:restriction_morphism}
    \Gamma_K\to \Gamma_{\FF_K}:=\Gal(k/\FF_K).
\end{equation}


 We obtain a short exact sequence
 \[
 1\to I_K\to\Gamma_K\to\Gamma_{\FF_K}\to 1.
 \] 
 A \emph{Frobenius element}  is an element $\phi\in \Gamma_K$ that maps to $\Frob_q:x\mapsto x^{q}$ in $\Gamma_{\FF_K}$, where $q=|\FF_K|$. 
 
  The \emph{Weil group} $W_K$ of $K$  is the subgroup of $\Gamma_K$ that fits in the exact sequence
 \[
 1\to I_K \to W_K \to Z\to 1,
 \]
 where $Z<\Gamma_{\FF_K}$ is the subgroup generated by $\Frob_q$.
 We define a topology on $W_K$ by endowing $I_K\subset \Gamma_K$ with the subspace topology and requiring that $I_K<W_K$ is an open subgroup. 

\begin{defn}
Let $V$ be a finite-dimensional $\CC$-vector space, endowed with the discrete topology.
\begin{itemize}
\item[(a)]
A \emph{Weil representation} of $K$ is a continuous group homomorphism 
\[
\rho:W_K\to \GL(V).
\]
\item[(b)] A Weil representation is called \emph{Frobenius semisimple} if the automorphism $\rho(\phi)\in \GL(V)$ is semisimple  for all Frobenius elements.
\end{itemize}
\end{defn}

Note that $I_K$  is compact. Therefore a group homomorphism  $\rho:W_K\to \GL(V)$ is a Weil representation if and only if there exists an open subgroup of $I_K$ whose image under $\rho$ is trivial. Recall that a Weil representation is called \emph{unramified} if the restriction to $I_K$ is trivial.  A Weil representation is called an \emph{Artin representation} if  it factors through the Galois group of a finite Galois extension $L/K$.

For a proof of the following result, we refer to \cite{BH}, Section 28.6.

\begin{prop} \label{prop:Weilirred}
Let $\rho$ be an irreducible Weil representation over $K$. Then it is finite dimensional, and there exists a $1$-dimensional unramified representation $\chi$ such that 
	\[
	\rho_0=\rho\otimes \chi^{-1}
		\] 
		 is an Artin representation. 
\end{prop}

A finite-dimensional Weil representation is semisimple if and only if it is Frobenius semisimple.  Proposition \ref{prop:Weilirred} implies therefore that every Frobenius-semisimple Weil representation $(\rho, V)$ can be written as
\begin{equation}\label{eq:Weildecomp}
V=\oplus_i V_i\otimes \chi_i,
\end{equation}
where $V_i\otimes \rho_i$ is the twist of an Artin representation by a $1$-dimensional unramified representation. 
To obtain a unique decomposition in \eqref{eq:Weildecomp}, we define an equivalence relation on $1$-dimensional unramified representations over $K$ by 
\[
\chi_i\sim \chi_j \quad \Leftrightarrow \quad \chi_i\circ \chi_j^{-1} \text{ factors through a finite extension.} 
\]
This condition is equivalent to requiring that $\chi_i(\phi)$ and $\chi_j(\phi)$ differ by a root of unity for a Frobenius element $\phi$.  
 The decomposition in \eqref{eq:Weildecomp} is unique if we let the sum run over a set of representatives of this equivalence relation.

\subsection{Stable reduction} \label{subsec:stable_reduction}

Let  $Y/K$ be a smooth, projective, and absolutely irreducible curve. For simplicity, we assume that the genus $g$ of $Y$ is at least $2$.

Let $K^\nr\subset\Kb$ denote the maximal subextension unramified over $K$, and let $L/K^\nr$ be a finite extension. Let $v_L$ denote the unique extension of the valuation $v_K$ to $L$ and $\OO_L$ the ring of integers with respect to $v_L$. This is a discrete valuation ring with residue field $k$. We let $Y_L$ denote the base change of $Y$ to $L$.

By the Semistable Reduction Theorem (\cite{DeligneMumford69}) the curve $Y_L$ has semistable reduction if the extension $L/K^\nr$ is sufficiently large. This means that there exists a flat and proper $\OO_L$-scheme $\Y$ with generic fiber $Y_L$ such that the special fiber $\Y_s:=\Y\otimes_{\OO_L}k$ is reduced and has at most ordinary double points as singularities. We may further assume that $\Y$ is the {\em stable model} of $Y_L$, i.e.\ all irreducible components of $\Y_s$ of geometric genus $0$ intersects the other components in at least $3$ points, where a point of self intersection counts as $2$ points. Under this additional assumption, the model $\Y$, and hence its special fiber $\Y_s$, are uniquely determined by $Y$ and the extension $L/K$. 

If we replace $L$ by some finite extension, the special fiber $\Y_s$ of the stable model remains unchanged. This fact has the following important consequences. Firstly, the $k$-curve $\Y_s$ only depends on the original curve $Y/K$, and not on the choice of $L/K$. We call it the {\em stable reduction} of $Y/K$. Secondly, we may assume that $L/K$ is a Galois extension.

Write $G:=\Gal(L/K)$ for the Galois group of $L/K$ and $I:=\Gal(L/K^\nr)\subset G$ for its inertia subgroup. We have a short exact sequence
\[
   1 \to I \to G \to \Gamma_{\FF_K} \to 1.
\]
Here we identify the Galois group of $K^\nr/K$ with $\Gamma_{\FF_K}$, the Galois group of $k/\FF_K$. The group $G$ acts naturally on the $k$-curve $\Y_s$; this action is $k$-semilinear, meaning that the induced action on the constant base field $k$ of $\Y_s$ is the natural action corresponding to the morphism $G\to\Gamma_{\FF_K}$. In particular, the action of the inertia group $I$ on $\Y_s$ is $k$-linear. 

\begin{prop}\label{prop:Gfaithful}
  There exists a unique minimal extension $L/K^\nr$ such that $Y_L$ has semistable reduction. The extension $L/K$ is Galois, and the action of $G=\Gal(L/K)$ on $\Y_s$ is faithful. 
\end{prop}	

\begin{proof}
Let $H\subset G$ denote the kernel of the morphism
\[
    G\to\Aut_{\FF_K}(\Y_s),
\]	
and let $L':=L^H$ be its fixed field.
Clearly, $H$ is contained in $I$, and so is finite. The quotient scheme $\Y':=\Y/H$ is again semistable (\cite[Cor.~2.3.3]{Raynaud3pt}), with special fiber $\Y'_s=\Y_s$. It follows that $L'/K$ is an extension with all the required properties, such that the action of $G':={\rm Gal}(L'/K)=G/H$ on $\Y'_s$ is faithful. This proves the existence part of the proposition. The uniqueness is clear. 
\end{proof}

From now on, we assume that the extension $L/K$ is chosen as in the proposition. It follows that the action of $G$ on the stable reduction $\Y_s$ is faithful. 

\subsection{Reduction of compact type} \label{subsec:compact_type}

For the rest of this article, we make the following additional assumption. 

\begin{assumption}  \label{ass:compact_type}
  We assume that the curve $Y/K$ has potential reduction \emph{of compact type}, i.e.\ that the graph of components of the stable reduction $\Y_s$ is a tree. 
\end{assumption}

It is well known that Assumption \ref{ass:compact_type} holds if and only if the Jacobian of $Y/K$ has potential good reduction (see e.g.\ \cite{SerreTate68}).

The semistable $k$-curve $\Y_s$ is connected and has arithmetic genus $g$. We let
\[
   \psi:\Yb\to\Y_s
\]
denote the normalization of $\Y_s$. Then $\Yb$ is a smooth projective curve over $k$, not necessarily irreducible. In fact, $\Yb$ is the disjoint union of the irreducible components of $\Y_s$, which are all smooth, by Assumption \ref{ass:compact_type}.
The action of $G={\rm Gal}(L/K)$ on $\Y_s$ induces an action on $\Yb$, which is faithful as well. From now on, we consider $G$ as a subgroup of $\Aut_{\FF_K}(\Yb)$ via this action. We call the natural map
\begin{equation} \label{eq:residual_action}
  \kappa:\Gamma_K \to G\subset\Aut_{\FF_K}(\Yb)
\end{equation}
the {\em residual action}.

\vspace{2ex}
Choose a prime number $\ell\neq p$. We have linear maps 
\begin{equation} \label{eq:cospecialization}
H^i_\et(Y_{\Kb},\QQ_\ell) \longleftarrow H^i_\et(\Y_s,\QQ_\ell) \stackrel{\psi^*}{\longrightarrow}
H^i_\et(\Yb,\QQ_\ell).
\end{equation}
The left arrow is the {\em cospecialization map}. Note that the group $\Gamma_K$ acts on all three vectors spaces, and both maps are $\Gamma_K$-equivariant. Moreover, the action on the vector spaces in the middle and on the right factor through the residual action $\kappa:\Gamma_K\to G$. Using Assumption \ref{ass:compact_type}, we obtain a much stronger statement.

\begin{prop}\label{prop:compact_type}
  \begin{enumerate}[(i)]
  \item
    Both arrows in \eqref{eq:cospecialization} are isomorphisms.
  \item
    The action of $\Gamma_K$ on $H^1_\et(Y_{\Kb},\QQ_\ell)$ also factors through the quotient $\Gamma_K\to G$.    
  \end{enumerate}
\end{prop}

\begin{proof}
Clearly, (ii) follows from (i). Statement (i) may be considered as a special case of \cite[Prop.~2.6]{superell_preprint}. Loc.~cit.~states that $H^1_\et(Y_{\Kb},\QQ_\ell)^{I_L}\simeq H^1(\Y_s, \QQ_\ell)$. The key point is that $I_L$ acts trivially on $H^1_\et(Y_{\Kb},\QQ_\ell)$ in the case that $Y_L$ has reduction of  compact type, since the sheaf of vanishing cycles is trivial. This is contained as an easier special case  in the arguments in \S\S2.5--2.8 of \cite{superell_preprint}. 
\end{proof}	

\bigskip
We use the notation
\[
    V_\ell := H^1_\et(Y_{\Kb},\QQ_\ell).
\]
The natural action of the Galois group $\Gamma_K$ on $V_\ell$ yields an $\ell$-adic Galois representation
\begin{equation} \label{eq:rho_ell}
   \rho_\ell:\Gamma_K\to \GL_{\QQ_\ell}(V_\ell),
\end{equation} 
which is our main object of interest. We also set
\[
   \bar{V}_\ell := H^1_\et(\Yb,\QQ_\ell),
\]
and let 
\[
   \overline{\rho}_\ell:G\to \GL_{\QQ_\ell}(\overline{V}_\ell)
\]
denote the $\ell$-adic representation of $G$.

\begin{cor} \label{cor:compact_type}
  We have a commutative diagramm
  \begin{equation} \label{eq:rho_ell_diagramm}
  \xymatrix{
      \Gamma_K \ar[r]^-{\rho_\ell} \ar[d]_{\kappa} 
              & \GL_{\QQ_\ell}(V_\ell) \ar[d]^{\cong} \\
      G \ar[r]^-{\overline{\rho}_\ell} 
              & \GL_{\QQ_\ell}(\overline{V}_\ell), 
  }
  \end{equation}
  in which the arrow on the right is a natural isomorphism.
\end{cor}

\begin{rem}\label{rem:rho_ell_computation}
The diagram \eqref{eq:rho_ell_diagramm} is our main tool to study the $\ell$-adic representation $\rho_\ell$. The assumption that $Y$ has potential reduction of compact type enables us to divide the study of $\rho_\ell$ into two separate parts. 
\begin{enumerate}[(a)]
\item
  The computation of the semistable reduction of the curve $Y$, together with  the residual action $\kappa$.
\item
  The computation of the representation $\bar{\rho}_\ell$.  	
\end{enumerate}	

For superelliptic curves of exponent prime to $p$, Part (a) has been essentially solved in \cite{superell}. For Part (b), one can use simple point counting and the trace formula from Theorem~\ref{thm:trace_formula} below.
\end{rem}

\subsection{The trace formula}\label{sec:trace}

In the rest of this section we look exclusively at the $k$-curve $\Yb$ and the subgroup $G\subset \Aut_{\FF_K}(\Yb)$. Our goal is to describe the $\ell$-adic cohomology groups
\[
    H^i_\et(\Yb,\QQ_\ell), \quad i=0,1,2,
\]
as $G$-representations. The dimensions of these vector spaces are well known: we have
\begin{equation}  \label{eq:dim_Hi}
\dim_{\QQ_\ell} H^i_\et(\Yb,\QQ_\ell) = 
\begin{cases}
d, & i=0,2,\\
2g_{\Yb}, & i=1, \\
0,  & i>2.
\end{cases}   
\end{equation}
Here $d$ is the number of connected components and $g_{\Yb}$ the arithmetic genus of $\Yb$. Note that 
\[
g_{\Yb} = \sum_i g_{\Yb_i},
\]
where $\Yb_i$ are the connected components of $\Yb$.

Recall that the group $G$ sits in a short exact sequence
\begin{equation} \label{eq:G_exact_sequence}
  1 \to I \to G \stackrel{r}{\longrightarrow} \Gamma_{\FF_K} \to 1,
\end{equation}
where $I:=G\cap\Aut_k(\Yb)$ is the {\em finite} subgroup of $k$-linear automorphisms of $\Yb$ and where $r:G\to\Gamma_{\FF_K}$ is induced by the action of $G$ on the constant base field $k$ of $\Yb$. Recall also that $\Gamma_{\FF_K}$ is a procyclic group, topologically generated by the arithmetic Frobenius element $\Frob_{\FF_K}$.
 
\begin{defn}
\begin{enumerate}[(i)]
\item	
  An element of $\phi\in G$ is called an $\FF_K$-Frobenius element if $r(\phi)=\Frob_{\FF_K}$.
\item
  An $\FF_K$-model of $\Yb$ is a smooth projective curve $\Yb_0$ over $\FF_K$, together with a $k$-linear isomorphism  
  \[
     \Yb\cong \Yb_0\otimes_{\FF_K}k.
  \]
  (We will usually identify $\Yb$ with $\Yb_0\otimes_{\FF_K}k$.)
\end{enumerate}  
\end{defn}

\begin{lem} \label{lem:phi-models}
	Let  $\phi\in G$ be an $\FF_K$-Frobenius element. Then there exists an $\FF_K$-model $\Yb_\phi$ of $\Yb$, unique up to unique isomorphism, such that
	\[
    	\phi = \Id_{\Yb_\phi}\otimes \Frob_{\FF_K}.
	\]
	The association $\phi\mapsto \Yb_\phi$ defines a bijection between $\FF_K$-Frobenius elements and $\FF_K$-models of $\Yb$, up to isomorphism.
\end{lem}

\begin{proof}
Since $G$ is a profinite group and $\Gamma_{\FF_K}$ is pro-cyclic with topological generator $\Frob_{\FF_K}$, there exists a unique section $s:\Gamma_{\FF_K}\to G$ of $r$ such that $s(\Frob_{\FF_K})=\phi$. 

An element $\phi_\gamma\in G\subset \Aut_{\FF_K}(\Yb)$ lifting an element $\gamma\in\Gamma_{\FF_K}$ can be identified with a $k$-linear isomorphism
\[
\phi_\gamma:\prescript{\gamma}{}{\Yb} \stackrel{\sim}{\longrightarrow} \Yb,
\]
where 
\[
\prescript{\gamma}{}{\Yb} := \Yb\times_{\Spec(\gamma)}\Spec(k)
\]
is the base change of $\Yb$ along the isomorphism $\Spec(\gamma):\Spec(k)\stackrel{\sim}{\to}\Spec(k)$. 
Using this identification, the section $s$ defines a Weil cocycle $(\phi_\gamma)_{\gamma\in\Gamma_{\FF_q}}$ for the $k$-variety $\Yb$, via $\phi_\gamma:=s(\gamma)$. We see that the statement of the lemma is just a reformulation of Weil descent (\cite{WeilDescent}, see also \cite[Ch.~V.20]{SerreGACC} and \cite[Chapter 1]{Huggins_thesis}). 	
\end{proof}

\begin{rem}
	In \S\ref{sec:DDmethod} we apply Lemma \ref{lem:phi-models} also to  Frobenius elements corresponding to finite extensions of $\FF_K$.
\end{rem}

\begin{thm} \label{thm:trace_formula}
	Let $\phi\in G$ be an $\FF_K$-Frobenius element, and let $\Yb_\phi$ be the corresponding $\FF_K$-model of $\Yb$. Then
	\[
    	\sum_{i=0}^2 (-1)^i \Tr(\phi^{-1}, H^i_\et(\Yb,\QQ_\ell))
	          = \abs{\Yb_\phi(\FF_K)}.
	\]	
\end{thm}

\begin{proof} (See also \cite[pp. 25-26]{KatzReview})
Let $F_{\Yb,q}^{\rm abs}:\Yb\to\Yb$ be the absolute $q$-Frobenius endomorphism (defined for any scheme of characteristic $p$). By the definition of the $\FF_q$-model $\Yb_\phi$, we have a factorization 
\begin{equation} \label{eq:trace_formula1}
	F_{\Yb,q}^{\rm abs} = \phi\circ F_{\Yb,q}^{\rm rel} = F_{\Yb,q}^{\rm rel}\circ\phi,
\end{equation}
where $F_{\Yb,q}^{\rm rel}:\Yb\to\Yb$ is the relative Frobenius endomorphism induced by the identification
\[
	\Yb = \Yb_\phi\otimes_{\FF_q} k.
\]
By \cite[Proposition 29.10]{MilneLEC}, the endomorphism $F_{\Yb,q}^{\rm abs}$ acts on $H^i_\et(\Yb,\QQ_\ell)$ as the identity. Therefore, the factorization \eqref{eq:trace_formula1} shows that
\begin{equation} \label{eq:trace_formula2}
	F^{\rm rel}_{\Yb,q}|_{H^i(\Yb,\QQ_{\ell})} = \phi^{-1}|_{H^i(\Yb,\QQ_{\ell})}.
\end{equation}
In view of \eqref{eq:trace_formula2} the statement of the theorem is now the standard trace formula of Grothendieck, see \cite[Proposition 27.3]{MilneLEC}.
\end{proof}

\begin{rem}
	With the notation of Theorem \ref{thm:trace_formula}: let $d$ denote the number of absolutely irreducible components of $\Yb_\phi$. Then 
	\[
	\Tr(\phi^{-1},H^i_\et(\Yb,\QQ_\ell)) = 
	\begin{cases}
	d, & i=0, \\
	dq, & i=2. 
	\end{cases}
	\]	
	To prove this, use that the absolutely irreducible components of $\Yb_\phi$ correspond to the irreducible components of $\Yb$ which are fixed by $\phi$. 
\end{rem}

\begin{cor} \label{cor:trace_formula}
	For an $\FF_K$-Frobenius element $\phi$ we have
	\[
	\Tr(\phi^{-1},\overline{V}_\ell) = d(q + 1) - \abs{\Yb_\phi(\FF_K)}.
	\]	
\end{cor}

\subsection{Local polynomials}

Using the trace formula from Corollary \ref{cor:trace_formula} we can compute the traces of all $\FF_K$-Frobenius elements $\phi$ of $G$ on $\overline{V}_\ell$. This is not quite enough to determine the representation $\overline{\rho}_\ell$; we also need the traces of the powers $\phi^m$, $m\geq 1$.  The obvious solution is to pass to the (unique) unramfied extension $K_m/K$ of degree $m$.

It is convenient to change and expand our notation a bit. We write $\FF_q$ instead of $\FF_K$, where $q$ is the cardinality of the field $K$, and similarly $\FF_{q^m}$ for the residue field of $K_m$, which is a field with $q^m$ elements. We also write $\Frob_{q^m}$ for the generator of $\Gamma_{\FF_{q^m}}={\rm Gal}(k/\FF_{q^m})$ given by $\alpha\mapsto\alpha^{q^m}$. An element $\phi\in G$ with $r(\phi)=\Frob_{q^m}$ is called a {\em $q^m$-Frobenius element}.

\begin{defn}
	Let $\phi$ be a $q$-Frobenius element, with $\FF_q$-model $\Yb_\phi$. Then  
	\[
	P(\phi^{-1}, T) := \det\left(1-T\phi|_{\overline{V}_\ell}\right)
	\]	
	is called the {\em local polynomial} of the model $\Yb_\phi$.
\end{defn}	

\begin{prop} \label{prop:local_polynomial}
\begin{enumerate}[(i)]
\item
  The polynomial $P(\phi^{-1},T)$ has integral coefficients and is of the form
  \[
		P(\phi^{-1},T) = 1 + \ldots + q^{g}T^{2g}
		\in\ZZ[T].
  \] 
\item
  Over the complex numbers the polynomial $P(\phi^{-1},T)$ decomposes as
  \[
		P(\phi^{-1},T) = \prod_{i=1}^{2g}(1-\alpha_iT),
  \]
  where the $\alpha_i$ are complex numbers with absolute value $\abs{\alpha_i}=q^{1/2}$.
\item
  For all $m\geq 1$ we have
  \[
     \sum_{i=1}^{2g} \alpha_i^m = \Tr(\phi^{-m}, \overline{V}_\ell). 
  \]
\end{enumerate}	
\end{prop}	

\begin{proof}
	This is a very special case of the Weil Conjectures, see \cite[pp. 26-27]{KatzReview}. (If $\Yb$ is irreducible and hence $\Yb_\phi$ absolutely irreducible, this is a famous theorem of Weil, \cite{WeilCourbes}). 
\end{proof}	

\begin{rem}
\begin{enumerate}[(i)]
\item
  Using Proposition \ref{prop:local_polynomial} (iii) and Corollary \ref{cor:trace_formula} we can compute $P(\phi^{-1},T)$ explicitly by counting the points of $\Yb_\phi$ over the fields $\FF_{q^m}$, for $m=1,\ldots,g$. 
\item
  It follows that $P(\phi^{-1},T)$ does not depend on the choice of the auxiliary prime number $\ell$.
\end{enumerate}		
\end{rem}

As a side product of our investigation, we obtain a self-contained proof of the following well known theorem. We include it here as we could not find a good reference. 

\begin{thm}\label{thm:compact_type} Let $Y/K$ be a curve with potential reduction of  compact type.
\begin{enumerate}[(i)]
\item
  There exists a unique Weil representation 
  \[
     \rho: W_K\to\GL(V)
  \]
  which induces the $\ell$-adic representation $\rho_\ell$, for all $\ell\neq p$.
\item
  The representation $\rho$ is semisimple.
\end{enumerate}
\end{thm}

\begin{proof}
The statements follow essentially from Theorems \ref{thm:trace_formula} and \ref{cor:compact_type}. Namely, by Corollary \ref{cor:compact_type} the $\ell$-adic representation is determined by $\rho_\ell$.  
The trace formula in Theorem \ref{thm:trace_formula} shows that the traces of the Frobenius elements in $\Aut_{\FF_K}(\Yb)$ acting on $\bar{V}_\ell$ are Weil numbers,  which are independent of $\ell$.
See also   \cite{Rohrlich}, proposition in Section 4, in the case of potential good reduction.
\end{proof}

\begin{rem}
  The analogous statement to Theorem \ref{thm:compact_type} in the case that $Y$ does not have reduction of compact type is also well known. The main complication is that one has to use {\em Weil--Deligne representations} instead of Weil representation. In a future article, we will discuss the computation of the Weil--Deligne representation of a curve $Y$, extending the approach of the current paper.  
\end{rem}

\subsection{Computing the Weil representation using the residual action}\label{sec:DDmethod}
We briefly  describe how to compute the Weil representation of a curve with potential reduction of compact type using only the residual action \eqref{eq:G_exact_sequence}. Our approach is a variant of the method of \cite{DD}. In \S\ref{sec:exa} we illustrate the method in three concrete cases.

Let $Y/K$ be a curve with potential  reduction of compact type. We assume that we are given a Galois extension $L/K^{\nr}$ over which $Y$ has stable reduction, together with a semistable  model $\Y$ over $\mathcal{O}_L$. We may assume that $L/K^{\nr}$ is minimal with this property. Proposition \ref{prop:Gfaithful} implies that $I=\Gal(L/K^{\nr})$ acts faithfully in the special fiber $\Y_s$. Write $\Yb$ for the normalization of $\Y_s$.

Let $q=|\FF_K|$. Lemma \ref{lem:phi-models} implies that there exists an $\FF_q$-model $\Yb_0$ of $\Yb$. Let $\phi_0$ be the $q$-Frobenius element corresponding to $\Yb_0$ under the correspondence from Lemma \ref{lem:phi-models}. 
 We may also  identify $I$ with a subgroup of $\Aut_k(\Yb)$.  
 We compute the local polynomial $P(\phi_0^{-1}, T)$ by counting points on $\Yb_0$. 
 The decomposition of the Weil representation $\rho$ of $Y$ in \eqref{eq:Weildecomp} is obtained by ordering the roots of the local polynomial $P(\phi_0^{-1}, T)$ in equivalence classes, where we consider two roots to be equivalent if their quotient is a root of unity.
 More concretely, every root $\gamma$ of the local polynomial defines an unramified character 
 \[
 \chi:W_K\to \CC^\times, \qquad \Frob_K\mapsto \gamma, 
 \]
 where $\Frob_K$ is a Frobenius element with $\kappa(\Frob_K)=\phi_0$.  In \cite{DD} it is explained how to reduce the calculation of the Weil representation to the case that $\rho=\rho_0\otimes \chi$, where $\rho_0$ is an Artin representation and $\chi$ is an unramified character, i.e.~to one of the factors in the decomposition \eqref{eq:Weildecomp}. This is  Step 4 in Section 2 of \cite{DD}. We refer to that paper or Section 3.4 of \cite{Do} for more details.

 We may restrict to the case that $\rho=\rho_0\otimes\chi$, where $\rho_0$ is an Artin representation and $\chi$ an unramified $1$-dimensional representation.
The Artin representation factors through a finite quotient 
\begin{equation}\label{eq:Gbar}
1\to I\to \overline{G}\to \langle \overline{\phi}_0\rangle\to 1
\end{equation}
of the group $G$ in \eqref{eq:G_exact_sequence}.

Let $\phi\in G$  be an $q^f$-Frobenius element. Write $\Yb_\phi$ for the corresponding model of $\Yb$ over $\FF_{q^f}$. The trace of $\rho_0(\phi)$ can be computed using Corollary \ref{cor:trace_formula}. If the representation $\rho_0$ is not rational, one has to be a bit careful with identifying elements in $K$ with complex numbers. 

For elements $g\in I$ that are not Frobenius elements, one uses that $\rho_0(\phi_0)$ has finite order $f$. Now $\phi_0^fg$ is a $q^f$-Frobenius element and the trace of $\rho_0(g)$ is the same as that of $\rho_0(\phi_0^fg)$, and can hence be computed. 

The advantage of this method is that once we are given one concrete model of $Y_L$ of compact type over a concretely given extension $L/K$, all other calculations can be done purely in positive characteristic. In \S\ref{sec:superred} we discuss how all steps can be made explicit in the case of superelliptic curve with potential reduction of compact type to characteristic $p$ in the case that $p$ does not divide the exponent of the superelliptic curve.

\section{Superelliptic curves} \label{sec:superell}

In this section, we introduce the class of curves we treat in this paper, namely  superelliptic curves. 

\subsection{Stable reduction of superelliptic curves}\label{sec:superred}

We recall some results from \cite{superell} on the  reduction of superelliptic curves. Let $K/\QQ_p$ be a finite extension.  A \emph{superelliptic curve} $Y/K$ is a smooth projective curve birationally determined by an equation of the form
\[
y^n = f(x),
\]
 where $f\in K[x]$ is a nonconstant polynomial. It is no restriction to assume that $f\in \OO_K[x]$.  We assume that $f$
 has no nontrivial factor that is an $n$th power in $K[x]$.
 Then the morphism
 \begin{equation}\label{eq:phidef}
 \pi:Y\to X:=\PP^1_K, \qquad (x,y)\mapsto x
 \end{equation}
 is branched at all zeros of $f$, and possibly at $\infty$.  

Let $L_0/K^{\nr}$ be the splitting field of $f$ and $S\subset L_0$ the set of roots of
$f$. Then we can write
\[
  f=c\prod_{\alpha\in S}(x-\alpha)^{a_\alpha},
\]
with $c\in K^\times$ and $a_\alpha\in\NN$. We impose the
following conditions on $Y$.

\begin{assumption} \label{fnass} 
\begin{itemize}
\item[(a)]
  We have $\gcd(n,a_\alpha\mid \alpha\in S)=1$.  
\item[(b)]  
  The genus $g=g(Y)$ of $Y$ is $\geq 2$.
\item[(c)]
  The exponent $n$ is $\geq 2$ and prime to $p$.
\end{itemize}
\end{assumption}

Part (a) of this assumption implies that $Y$ is absolutely irreducible. Using the Riemann-Hurwitz formula, it is then easy to compute the genus $g$ in terms of $n$ and $(a_\alpha)_{\alpha\in S}$.
Part (c) is the crucial assumption needed for the method of \cite{superell} to compute the semistable reduction.


Note that the valued field $L_0$ is henselian, the residue field $k$ of $L_0$ is algebraically closed of characteristic $p$, and $n$ is prime to $p$ by Assumption \ref{fnass} (c). It follows that there exists a unique extension $L/L_0$ of degree $n$, which is totally (and tamely) ramified.

\begin{prop}\label{prop:red}
Let $Y/K$ be a superelliptic curve satisfying Assumption \ref{fnass}. 
\begin{itemize}
\item[(a)] 
   The curve $Y$ has semistable reduction over $L$.
\item[(b)]
  Let $\Y$ be the stable model of $Y_L$ and $\Yb$ the normalization of its special fiber $\Y_s$. 
  
  The  curve $\Yb$ is the disjoint union of (not necessarily connected) superelliptic curves
  \[
    \Yb_i:\; {y_i}^{n}=\overline{f}_i({x_i}), \qquad \overline{f}_i\in k[{x_i}]
  \]
  over $k$. The map $\pi:Y\to X$ induces the maps
  \[
    \overline{\pi}_i: \Yb_i\to \Xb_i:=\PP^1_k, \quad ({x_i}, {y_i})\mapsto x_i.
  \]
For each $i$, the coordinate $x_i$ may be written as $x=A_ix_i$ for some $A_i\in \GL_2(K)$.
\end{itemize}
\end{prop}

\begin{proof}
	Statement (a) is  a special case of \cite[Corollary 4.6]{superell}.
	Statement (b)  is part of the statement of \cite[Proposition 4.5]{superell}. 
\end{proof}
 
 In \cite[Proposition 4.2]{superell} it is explained how to determine the coordinates $x_i$. 

\subsection{The case of potential good reduction}

The results from \cite{superell}, summarized by Proposition \ref{prop:red}, allow an explicit description of the action of the inertia group $I=\Gal(L/K^\nr)$ on $\Yb$. In this subsection we explain this in detail under an additional assumption which guarantees that the curve $Y$ has potential good reduction. See Example \ref{exa:compact_type} for an example where this condition is not satisfied. 

\begin{assumption} \label{ass:equidistant}
  The set of roots $S$ is {\em equidistant}, i.e.\ for all pairs $\alpha_1,\alpha_2\in S$ of distinct roots, the valuation $v_{L_0}(\alpha_1-\alpha_2)$ is the same.
\end{assumption}	

\begin{prop} \label{prop:good_reduction}
  If Assumption \ref{ass:equidistant} holds, than $Y$ has potential good reduction. In particular, the curve $\Yb$ is absolutely irreducible.
\end{prop}	

\begin{proof}
This is also just a special case of \cite[Proposition 4.2]{superell}. For later use, we recall the main argument in this case. As a result we obtain an explicit description of the special fiber $\Yb$. 

Write $S=\{\alpha_1,\ldots,\alpha_r\}$, with $r=\abs{S}$, and $a_i:=a_{\alpha_i}$. Set 
\begin{equation}\label{eq:coord_good}
    x_0 := \frac{x-\alpha_1}{\alpha_2-\alpha_1} \in L[x].
\end{equation}
We write $f$ as a polynomial in $x_0$,
\[
    f = \sum_i c_i x_0^i,
\]
and let 
\[
    m:= \min_i v_{L_0}(c_i) 
\]
denote the Gauss valuation of $f$ with respect to $x_0$. Here $v_{L_0}$ is the valuation on $L_0$, normalized such that $v_{L_0}(L_0^\times)=\ZZ$. Let $\pi_{L_0}$ be a prime element of $L_0$. We let $\overline{f}\in k[x_0]$ denote the reduction of $\pi_{L_0}^{-m}f$, i.e.\
\[
    \overline{f}:=\sum_i \overline{c}_i x_0^i, \quad \text{with}\; \overline{c}_i := \overline{\pi_{L_0}^{-m}c_i}\in k.
\]
Then Assumption \ref{ass:equidistant} implies that $\deg(\overline{f})=\deg(f)$ and that $\overline{f}$ has the $r$ distinct roots
\[
    \overline{\alpha}_i := \overline{\pi_{L_0}^{-m'}\frac{\alpha_i-\alpha_1}{\alpha_2-\alpha_1}}, \quad i=1,\ldots,r.
\]
Moreover, the multiplicity of $\overline{\alpha}_i$ in $\overline{f}$ is equal to $a_i$, the multiplicity of $\alpha_i$ in $f$. 

By the definition of the extension $L/L_0$ there exists a prime element $\pi_L$ of $L$ such that $\pi_L^n=\pi_{L_0}$. We set
\[
    y_0 := \pi_L^{-m}y.
\]
Now by construction the curve $Y_L$ can be written as the superelliptic curve
\[
    Y_L:\; y_0^n = \pi_{L_0}^{-m} f.
\]
Let $\Y/\OO_L$ be the `obvious model' of $Y_L$ corresponding to the choice of the coordinates $x_0,y_0$. The proof of \cite[Proposition 4.5.(2)]{superell} shows that $\Y$ is smooth and that the special fiber $\Yb:=\Y_s$ of $\Y$ is the superelliptic curve 
\[
   \Yb:\; y_0^n = \overline{f}(x_0)
\]
over $k$. Note that $\Yb$ has the same degree and the same ramification type as $Y$ (determined by the multiplicities $a_i$), which directly shows that $g(\Yb)=g(Y)$. 
\end{proof}	

\begin{rem}\label{rem:good_red}
  Assume that the two roots $\alpha_1,\alpha_2$ lie in a subextension $L_1\subset L_0$ which is totally ramified over $K$, and hence has residue field $\FF_K$. Then we may also assume that the prime element $\pi_{L_0}$ lies in $L_1$. It follows that $\overline{f}\in\FF_K[x_0]$. Therefore, the choice of the coordinates $x_0,y_0$ in the proof of Proposition \ref{prop:good_reduction} corresponds to an $\FF_K$-model $\Yb_0$ of $\Yb$. 
\end{rem}

Recall from \S\ref{sec:DDmethod} that the action of the  inertia group $I_K<W_K$ on $H^1_\et(Y_{\overline{K}}, \QQ_\ell)$ factors through a finite quotient group $I$. This group acts $k$-linearly on $\overline{Y}$.  By abuse of notation, we also call the finite group $I$  `inertia group'.  We start by making some general remarks on this action.

Let $H\subset\Aut_k(\Yb)$ denote the subgroup of automorphisms $\sigma:\Yb\to\Yb$ of the form
\begin{equation}\label{eq:Hdef}
\sigma^*(x) = x, \quad \sigma^*(y) = \zeta y,
\end{equation}

\begin{lem} \label{lem:superell_automorphisms}
	\begin{itemize}
\item[(a)]	The inertia group $I$ is contained in the normalizer of $H$ in $\Aut_k(\Yb)$. 
Moreover, it is a semidirect product $I=P\rtimes C$, where $P$ is an elementary abelian $p$-group and $C$ is cyclic of order $m$, where $m$ is prime to $p$ and divisible by $n$.
\item[(b)] We have that
\[
\tau^\ast(x)=x+\xi, \quad \tau^\ast(y)=y,
\]
for every $\tau\in P\cap I$. The group $C\cap I$  is cyclic of order $m$, generated by an automorphism $\sigma$ such that
  \[
      \sigma^*(x) = \zeta^{n}x, \quad \sigma^*(y) = \zeta y,
  \]
  where $\zeta\in k$ has order $m$.
\end{itemize}
\end{lem}

\begin{proof}	
  The first statement in (a)   follows immediately from the fact that the map $\pi$ in equation \eqref{eq:phidef} is defined over $K$.  Part (b)  and the second statement in (a) are contained in \cite[Proposition 5.5.]{superell}.
\end{proof}


Lemma \ref{lem:superell_automorphisms} imposes a strong restriction on the inertia group $I$. In the Examples \ref{exa:Picard1} and \ref{exa:Picard2} the bound from Proposition \ref{prop:GalMax} below is attained.

\begin{prop} \label{prop:GalMax}
Let $Y/K$ be a superelliptic curve with potential good reduction.  We assume, for simplicity,  that  $\pi:Y\to X$ is branched at $x=\infty\in X\simeq \PP^1_K$. Let $e+1$ be the number of branch points of 
$\pi$.   Then the cardinality of the wild inertia group $P$ from Lemma \ref{lem:superell_automorphisms} is bounded by $e$.
\end{prop}

\begin{proof}
In the situation of the proposition, the map $\pi$ reduces to a finite flat map
\[
\overline{\pi}:\Yb\to \Xb
\]
between smooth projective curves. The branch points of $\pi$ specialize to pairwise distinct points on $\Xb$, which are exactly the branch points of $\overline{\pi}$. It is no restriction to assume that $\infty$ specializes to $x=\infty\in \Xb$.

The (finite) inertia group $I$ acts $k$-linearly on $\Yb$. 
By Lemma \ref{lem:superell_automorphisms}.(a) the subgroup $H$, defined in \eqref{eq:Hdef}, is contained in the center of $I$. The quotient $I/H$ acts on $\Yb/H=\Xb.$ In particular, the wild part $P$ of the inertia group acts on $\Xb$. 
Since $\infty\in X$ is $K$-rational, it is a fixed point of $I/H$.  We conclude from Lemma \ref{lem:superell_automorphisms}.(b) that $P$ acts freely on the $e$ branch points of $\overline{\pi}$ different from $\infty$. The statement follows.
\end{proof}

\subsection{Models of superelliptic curves}\label{sec:supermodels}
In Lemma \ref{lem:phi-models} we described a correspondence between models of a curve over a finite field and Frobenius elements. 
In this section, we describe the  model corresponding to a given Frobenius element explicitly in the case of superelliptic curves. 

In this section, we let
$\Yb_0/\FF_q$ 
be a, not necessary connected, superelliptic curve given by an equation
\begin{equation} \label{eq:superell_model0}
\Yb_0:\; y^n = f(x),
\end{equation}
where $n$ is prime to $p=\Char(k)$, and $f\in \FF_q[x]$ is nonconstant. Set $\Yb:=\Yb_0\otimes_{\FF_q} k$.
 By definition, $\Yb$ is the smooth projective model of the affine curve given by the equation \eqref{eq:superell_model0}. In particular, the ring of rational functions $F_{\Yb}$ is the ring of fractions of the affine $k$-algebra $k[x,y\mid y^n=f(x)]$. 
Let $\phi_0$ denote the Frobenius element corresponding to the $\FF_q$-model $\Yb_0$ of $\Yb$ via the correspondence from Lemma \ref{lem:phi-models}. It is determined by the following conditions:
\begin{equation} \label{eq:superell_model1}
   \phi_0^*|_k = \Frob_q, \quad \phi_0^*(x) = x, \quad \phi_0^*(y) = y.
\end{equation}

In this section we explain how to explicitly construct the model $\Yb_\phi$ corresponding to an arbitrary $q$-Frobenius element. For our purpose, it is enough to know how to do this for Frobenius elements of the form $\phi=\phi_0g$, where $g\in I\subset \Aut_k(\Yb)$ is an element of the group $I$ described in Lemma \ref{lem:superell_automorphisms}. It is easy to see that the corresponding models are precisely the `superelliptic models' of $\Yb$, i.e. $\FF_q$-models of $\Yb$ that  are defined by an equation of the form \eqref{eq:superell_model0}.  

\begin{exa} \label{exa:superell_model_tame}
Let $\sigma\in I$ be an element of order prime to $p$. By Lemma \ref{lem:superell_automorphisms}.(b) $\sigma$ is conjugate to an element of the cyclic subgroup $C$. This means that, up to a change of the coordinate $x$, we may assume that
\begin{equation} \label{eq:superell_model_tame1}
   \sigma^*(x)=\xi x, \quad \sigma^*(y) = \zeta y,
\end{equation}
where $\xi,\zeta \in k^\times$ are such that
\begin{equation} \label{eq:superell_model_tame2}
  f(\xi x) = \xi^m f(x), \quad \zeta^n=\xi^m,
\end{equation} 
for a certain integer $m$. 
To construct the model $\Yb_\phi$ with $\phi=\phi_0\sigma$ we try to find new coordinates that are invariant under $\phi$. Our Ansatz is to write
\[
   x_1=\alpha x, \quad y_1=\beta y,
\]
with $\alpha,\beta\in k^\times$ to be determined. A brief calculation, using \eqref{eq:superell_model_tame1} and \eqref{eq:superell_model_tame2}, shows that
\[
  \phi^*(x_1) = \alpha^{q-1}\xi x_1, \quad 
  \phi^*(y_1) = \beta^{q-1}\zeta y_1.
\]
Therefore, we choose $\alpha,\beta\in k^\times$ satisfying
\begin{equation} \label{eq:superell_model_tame3}
  \alpha^{q-1} = \xi^{-1}, \quad \beta^{q-1} = \zeta^{-1}.
\end{equation}
With this choice, we obtain the $\FF_q$-model 
\begin{equation} \label{eq:superell_model_tame4}
  \Yb_\phi:\; y_1^n = f_1(x_1) := \beta^n f(\alpha^{-1}x).
\end{equation}
We invite the reader to check that $f_1\in\FF_q[x_1]$, using \eqref{eq:superell_model_tame2} and \eqref{eq:superell_model_tame3}.
\end{exa}

\begin{exa} \label{exa:superell_model_wild}
  Now let $\tau\in P$ be an element of $I$ whose order is divisible by $p$. It follows from Lemma \ref{lem:superell_automorphisms}.(b) that $\tau$ is of the form
  \[
      \tau^*(x) = x+\xi, \quad \tau^*(y) = \zeta y,
  \]  	
  where $\xi,\zeta\in k$ are such that
  \begin{equation} \label{eq:superell_model_wild1}
      f(x+\xi) = f(x), \quad \zeta^n = 1.
  \end{equation}
  The first condition in \eqref{eq:superell_model_wild1} holds if and only if $f$ is an additive polynomial such that $f(\xi)=0$.
  
  Set $\phi:=\phi_0\tau$. In order to find the model $\Yb_\phi$ we use the Ansatz
  \begin{equation} \label{eq:superell_model_wild2}
     x_1 = x + \alpha, \quad y_1 = \beta y.
  \end{equation}
  One checks that these coordinates are $\phi$-invariant if and only if
  \begin{equation} \label{eq:superell_model_wild3}
    \alpha^q-\alpha + \xi = 0, \quad \beta^{q-1} = \zeta^{-1}.
  \end{equation}
  With $\alpha, \beta$ chosen as in \eqref{eq:superell_model_wild3} we find the model
  \begin{equation} \label{eq:superell_model_wild4}
    \Yb_\phi:\; y_1^n = f_1(x) := \beta^n f(x_1-\alpha).
  \end{equation}
  Again, it is a nice exercise to check that $f_1\in\FF_q[x]$. 
\end{exa}

\section{Examples} \label{sec:exa}

We present three examples of semisimple Weil representations $\rho$  coming from superelliptic curves with potential reduction of compact type. More details on the first two examples can be found in \cite{Do}.  

We apply the strategy outlined in Section \ref{sec:DDmethod}. In the examples we give, the Weil representation is of the form $\rho=\rho_0\otimes \chi$, where $\rho_0$ is an Artin representation and $\chi$ an unramified character. As explained in Section \ref{sec:DDmethod}, we may reduce to this case. This was already observed in \cite{DD}. 
The unramified character $\chi$ is determined by the choice of a root of the local polynomial $P(\phi_0^{-1}, T)$ in the notation of Section \ref{sec:DDmethod}.
For concrete calculations, the most challenging part is calculating the Artin representation $\rho_0$, which factors through  a finite group $\overline{G}$. Our examples illustrate that the group $\overline{G}$ can be quite large. In fact, Example \ref{exa:Picard1} was chosen so that the group $\overline{G}$ is as large as possible for a Picard curve with potential good reduction to characteristic $p\neq 3$.  

 With our method, one can compute the trace of $\rho_0(g)$ for all $g\in \overline{G}$ by point counting on a suitable twist of $\overline{Y}$, as explained in Section \ref{sec:DDmethod}.
 In the concrete examples we discuss below, we simplify the calculation by using additional information. This is helpful when doing calculations by hand, as it reduces the amount of point counting one has to do. We briefly describe the two types of arguments we use.

 If the character table of the group $\overline{G}$ is known, it suffices to identify the character of $\rho_0$ as a sum of irreducible representations. The group $\overline{G}$ fits in an exact sequence \eqref{eq:Gbar}. Therefore one could use the facts of representations theory of semi-direct products, see for example \cite{Mackey}. 
  In the first and last example, the group $\overline{G}$ is a semi-direct product with an abelian normal subgroup $I$, and we can use a more elementary version of these results that can be found in Section 8.2 of \cite{Serre_RepLin}. In this case, it is easy to find the irreducible representations of $\overline{G}$ from those of $I$ and $\langle \overline{\phi_0}\rangle$.
   In the second example, the normal subgroup $I$ is no longer abelian. However, in this case it is easy to see that the representation $\rho_0$ is irreducible, and it is not necessary to use the full character table of the group to characterize the representation.

A further ingredient we use in the examples below is that $I$ acts on the curve $\overline{Y}$ via explicitly known automorphisms; this is the residual action described in \eqref{eq:residual_action}. From this action, one can compute the genus of $g(\overline{Y}/H)$, which is useful, since we have
\[
\dim_\CC V^H=2g(\overline{Y}/H),
\]
for subgroups $H<I$ of $I$. This is a special case of a well-known fixed point formula, see e.g.~\cite[Proposition 1.3]{Ellenberg}.

\begin{exa}\label{exa:Picard2} 
 Consider the smooth projective curve $Y/\QQ_2$ defined by the affine equation
\[
Y:\; y^3=x^4+2x^3+2=:f(x).
\]
As in the proof of Proposition \ref{prop:good_reduction}, we let $L_0/K^{\nr}$ be the splitting field of $f\in K^{\nr}[x]$.  Choose two distinct roots  $\alpha_1, \alpha_2\in L_0$ of $f$. We define coordinates $x_0, y_0$ for $Y$ by
\begin{equation}\label{eq:coordPicard2}
x=(\alpha_2-\alpha_1)x_0+\alpha_1, \qquad y=\beta^4 y_0,
\end{equation}
where $\beta^3=\alpha_2-\alpha_1.$  Define $L=L_0(\beta)$ and fix a primitive $3$rd root of unity $\zeta_3\in L$. 
 Then $L/K^{\nr}$ satisfies the conditions in Proposition \ref{prop:red}.(a). The extension $L/K^{\nr}$ is Galois, with Galois group $I \simeq C_2^2\times C_3$. The coordinates $x_0, y_0$ define a smooth model $\mathcal{Y}$ of $Y_L$. 
 Its reduction   is given by 
 \begin{equation}\label{eq:specialfiber2}
 	\Yb: \; {y}_0^3={x}_0^4+{x}_0.
 \end{equation}

 As explained in \S \ref{subsec:stable_reduction}, the inertia group $I$ acts as $k$-linear automorphisms on $\Yb$. We may choose generators $\tau_1, \tau_2, \sigma$ of $I$ acting on $\Yb$ as
 \begin{equation}\label{eq:Picard2_group1}
 \begin{split}
 \tau^\ast_1({x}_0, y_0)&=(x_0+1, y_0),\\
  \tau^\ast_2(x_0, y_0)&=(x_0+\zeta_3, y_0),\\
 \sigma^\ast(x_0, y_0)&=(x_0, \zeta_3y_0),
 \end{split}
 \end{equation}
where $\zeta_3\in \FF_4$ is the reduction of $\zeta_3\in L$.

 We write $\Yb_0$ for the $\FF_2$-model of $\Yb$ defined by  \eqref{eq:specialfiber2} and let $\phi_0\in W_K$ be the $2$-Frobenius element corresponding to $\Yb_0$ under the correspondence from Lemma \ref{lem:phi-models}.  By point counting over extensions of $\FF_2$ we find
 \begin{equation}\label{eq:localpoly2}
 P(\phi_0^{-1}, T)=8T^6+1.
 \end{equation}
Since all roots of $	P(\phi_0^{-1}, T)$ differ from each other by a root of unity, we conclude that
 \[
 \rho_0:=\rho\otimes\chi^{-1}:  W_K\to \GL(V)
 \]	
 is an Artin representation. Here $\chi$ is the $1$-dimensional unramified representation that sends $\phi_0$  to $1/\sqrt{-2}$.

By considering the roots of the local polynomial we see that  the eigenvalues of $\rho_0(\phi_0)$ are exactly the $6$-th roots of unity, each with multiplicity one. In particular, $\rho_0(\phi_0)$ has order $6$. We conclude that the Artin representation factors through the finite group
\[
 1\to I\to \overline{G}\to \langle \overline{\phi}_0\rangle \to 1,
\] 
 where $\overline{\phi}_0$ is an element of order $6$. Note that $\overline{G}$ is the Galois group of a subextension of $L/K$. Using that $\tau_2, \sigma\in \Aut_k(\Yb_0)$ are defined over $\FF_4$, but not over $\FF_2$, it follows  that the relations in this group are 
 \begin{equation}\label{eq:Picard2_group}
 \overline{\phi}_0\tau_1\overline{\phi}_0^{-1}=\tau_1, \quad
 \overline{\phi}_0\tau_2\overline{\phi}_0^{-1}=\tau_1\tau_2, \quad
 \overline{\phi}_0\sigma\overline{\phi}_0^{-1}=\sigma^{-1}.
 \end{equation}
 For example using \cite{Miller}, we identify this group as (72,30) in the library of small groups. A character table of this group can be found on Tim Dokchitser's website \href{https://people.maths.bris.ac.uk/~matyd/GroupNames/61/C3xC3sD4.html}{GroupNames.org} (\cite{GroupNames}).

 \bigskip
 We first illustrate how to calculate the trace of $\rho_0(g)$ using the strategy of Section \ref{sec:DDmethod} for various $g\in\overline{G}$. This works slightly different, depending on whether $g\in I$ or not. We perform the calculation for one group element in each case. 
 
 We first consider the $2$-Frobenius element $\phi_2:=\phi_0\tau_2\not\in I$. We apply the method from Example \ref{exa:superell_model_wild} to compute the twist  $\Yb_2$ of $\Yb_0$ corresponding to  $\phi_2$. 
 We define new coordinates $(x_2, y_2):=(x_0+\zeta_3, y_0)$. 
 From \eqref{eq:superell_model_wild4} we find the twist
 \[
 \Yb_2:\; y_2^3=x_2^4+x_2+1.
 \]
 By counting points on $\Yb_2$, we find the local polynomial
 \[
 P(\phi_2^{-1}, T)=1+4T^2+8T^4+8T^6=(2T^2+1)(4T^4+2T^2+1).
 \]
 We conclude that the eigenvalues of $\phi_2$ multiplied by $\sqrt{-2}$  are exactly the $4$ primitive $12$-th roots of unity, together with $\pm 1$. We conclude that the trace of $\rho_0(\phi_2)$ is $0$.

 We next consider the group element $g=\tau_1\in I$. The element $g$ is not a Frobenius element, so we can not apply Lemma \ref{lem:phi-models} directly.  We apply the following trick from \cite{DD}. We consider the $2^6$-Frobenius element $\phi_0^6$. We have already seen that $\rho_0(\phi_0^6)$ is the identity. It follows that the trace of $\rho_0(g)$ is equal to the trace of the Frobenius element $\rho_0(\phi_0^6g)$, which we can compute by point counting as before.

 Applying the method of Example \ref{exa:superell_model_wild} once more, we find that the twist $\Yb_4$ corresponding to $\phi_4:=\phi_0^6\tau_1$ is
 \[
 \Yb_4:\; y_4^3=x_4^4+x_4+\zeta_3^2,
 \]
 where $(x_4, y_4)=(x_4+c_4, y_0)$ with $\min_{\FF_4}(c_4)=x^2+\zeta_3^2x+1$. 
 The corresponding  local polynomial over $\FF_{2^6}$ is:
 \[
 P(\phi_4^{-1})=(8T+1)^4(8T-1)^2.
 \]
 Therefore the eigenvalues of $\rho_0(\phi_0^6\tau_1)$, and hence of $\rho_0(\tau_1)$, are $1$ (with multiplicity $2$) and $-1$ (with multiplicity $4$). For $\rho_0(\phi_0^6\tau_2)$ we find exactly the same result. The eigenvalues of  $\rho_0(\phi_0^6\sigma)$ are  $\omega, \omega^2$ (each with multiplicity one) and $-\omega, -\omega^2$  (each with multiplicity two). Here $\omega\in \CC$ is a primitive $3$rd root of unity. We stress that the choices of the primitive roots of unity $\zeta_3\in L$ and $\omega\in \CC$ are independent of each other. In  \cite[Example 4.3.4]{Do} the character of $\rho_0$ is computed by applying this approach to further elements of $\overline{G}$.

 \bigskip 
 Rather than computing the trace of $\rho_0(g)$ for many more elements of $\overline{G}$, we explain how much further information we really need to determine the representation $\rho_0$. First note that $\overline{G}$ has $12$ irreducible representations of dimension $1$ and $15$ of dimension $2$. This follows directly from the character table, but also from Proposition 25 in Section 8.2 of \cite{Serre_RepLin}.  
 Every $1$-dimensional irreducible representation factors through a cyclic quotient of $\overline{G}$, hence its kernel contains an abelian subgroup $A$ of cardinality at least $12$. Using the group action \eqref{eq:Picard2_group} and the Riemann--Hurwitz formula, one checks that we have $g(\overline{Y}/A)=0$ for every subgroup $A$ that is the kernel of a $1$-dimensional representation.   
We conclude that  that $\rho_0$ does not have a $1$-dimensional subrepresentation. Hence $V$ is the sum of three $2$-dimensional irreducible representations. 

 One of these is easily identified. Write $V_1=V^{\langle \tau_1\rangle}$ for the fixed space of $\tau_1$. This is a rational representation  on which $\tau_2$ acts as $-\text{Id}$. There is the unique such irreducible representation of dimension $2$.

The other two irreducible subrepresentation are  $V_2=V^{\langle \tau_2\rangle}$ and $V_3=V^{\langle \tau_1\tau_2\rangle}$. They are complex conjugate.  The trace of the restriction of  $\rho_0(\tau_2)$ (resp.~$\rho_0(\sigma)$) to $V_2$ and $V_3$ is $0$ (resp.~$-1$).  There are exactly two pairs of irreducible representations satisfying the requirements. To decide which ones  occur, it suffices to compute the action of a suitable element of order $6$, for example $g=\phi_0\sigma\tau_2$.  One computes that $\rho_0(\phi_0\sigma\tau_2)$ has eigenvalues $1, -\omega, -\omega^2$ (each with multiplicity $2$). This determines the representation completely, see \cite[Example 4.3.4]{Do}.
 	\end{exa}

\begin{exa}\label{exa:Picard1}
 Consider the smooth projective curve $Y/\QQ_2$ defined by the affine equation
	\[
	Y:\; y^3=x^4+2x+2.
	\]
Let $L_0/K^{\nr}$ be the splitting field of $f$ and $L=L_0(\beta)$, where $\beta^4=\alpha_2-\alpha_1$ for two different roots $\alpha_1, \alpha_2\in L_0$ of $f$. The extension $L/K^{\nr}$ is Galois with Galois group $I_2\simeq C_2^2\rtimes C_9$. 
	
	As in Example \ref{exa:Picard2}, $Y_L$ has good reduction over $L$. The smooth model is given by coordinates $(x_0, y_0)$ that are defined similar to \eqref{eq:coordPicard2}.  Its reduction  is the smooth projective curve given by
	\begin{equation}\label{eq:specialfiber1}
	\Yb: \; y_0^3=x_0^4+x_0.
	\end{equation}
 As in Example \ref{exa:Picard2}, the inertia group $I_2$ acts $k$-linearly on $\Yb$.  We may choose generators $\tau_1, \tau_2, \psi$
 of $I_2<\Aut_k(\Yb)$, where $\tau_1, \tau_2$ are as in \eqref{eq:Picard2_group1} and
 \[
 \psi(x_0, y_0)=(\zeta_9^3x_0, \zeta_9y_0). 
	\]
Here $\zeta_9\in\FF_{2^6}$ is a primitive $9$th root of unity that satisfies $\zeta_9^3=\zeta_3$ for the fixed $3$rd root of unity $\zeta_3\in \FF_4$ used in \eqref{eq:Picard2_group1}. With this choice, we have that $\psi^3=\sigma$ for $\sigma$ as in \eqref{eq:Picard2_group1}. 

Since the $\FF_2$-model  $\Yb_0$, together with the automorphisms induced by $\langle \tau_1, \tau_2\rangle$,  is exactly the same as in Example \ref{exa:Picard2}, the trace of $\rho_0(\tau_i)$ and of $\rho_0(\overline{\phi}_0)$ is exactly the same as what we computed in the previous example.

	We write $\Yb_0/\FF_2$ for the $\FF_2$-model given by the equation \eqref{eq:specialfiber1} and write $\phi_0$  for the corresponding $2$-Frobenius element.
As in Example \ref{exa:Picard2}, the Weil representation may be written as
$\rho=\rho_0\otimes \chi$, where $\rho_0$ is an Artin representation and $\chi$ the unramified one-dimensional representation that sends $\phi_0$ to $1/\sqrt{-2}$. The Artin representation factors through
\[
1\to I\to \overline{G}\to \langle\overline{\phi}_0\rangle \to 1,
\]
where $\overline{\phi}_0$ has order $6$, as before. 

To compute the trace of $\rho_0(\psi)$, we apply the same trick as in Example \ref{exa:Picard2}.
The model $\Yb_5$ of $\Yb$ corresponding to the $2^6$-Frobenius element $\phi_5:=\phi_0\psi$ is given by 
\[
\Yb_5:\; y_5^3=ux_5^4+x_5,
\]
where $u\in \FF_{2^6}$ satisfies $u^7=\zeta_9$. By counting points on $\Yb_5$ over extensions of $\FF_{2^6}$ we find that the eigenvalues of $\rho_0(\psi)$ are the $6$ primitive ninth roots of unity, each with multiplicity one. Details on the calculation can be found in \cite[Example 4.3.6]{Do}.

We claim that the Artin representation $\rho_0$ is irreducible.   To see this, we consider the restriction of $\rho_0$ to the subgroup $H=\langle \psi, \overline{\phi}_0\rangle<\overline{G}$.  The previous calculation implies that $\rho_0(\psi)$ has order $9$. In Example \ref{exa:Picard2} we already showed that $\rho(\overline{\phi}_0)$ has order $6$. In the group $H$ we have the relation
\[
\overline{\phi}_0 \psi \overline{\phi}_0^{-1}=\psi^a, 
\]
for some element $a\in (\ZZ/9\ZZ)^\ast$ of order $6$. It follows that the restriction $\rho_0|_H$, and hence $\rho_0$, is irreducible.   In fact, $\rho_0|_H$ is the unique irreducible representation of dimension $6$, see  Proposition 25 in Section 8.2 of \cite{Serre_RepLin}.  The Artin representation $\rho_0$ is now completely determined.
	\end{exa}

The final example treats a Picard curve that does not have potential good reduction to characteristic $p=2$.
	
	\begin{exa}\label{exa:compact_type}
Consider the smooth projective curve $Y/\QQ_2$ defined by the affine equation
	\[
	Y:\; y^3=f(x)=x^4+1.
	\]
	By Proposition \ref{prop:red} this curve has  semistable reduction over $L:=\QQ^{\nr}_2(\zeta_8, \sqrt[3]{2})$, with $I:=\Gal(L/K^{\nr})\simeq C_2^2\times C_3$. Note that the roots of $f$ are  not equidistant in the sense of Assumption \ref{ass:equidistant}.
	
  Write $\Y$ for the semistable model of $Y_L$ and $\Yb$  for the normalization of the special fiber $\Y_s$. The curve $\Yb$ consists of a chain of three irreducible components, each with genus $1$.  We choose a primitive $8$th root of unity $\zeta\in \QQ_2^{\nr}$ and define $3$ coordinates on $X=\PP^1_x$ by
\[
x=(\zeta^3-\zeta)x_1+\zeta, \qquad  x=(\zeta^5-\zeta)x_2+\zeta, \qquad x=(\zeta^7-\zeta)x_3+\zeta^3.
\]
The coordinates $x_i$ are defined similar to \eqref{eq:coord_good}. However, each triple  of  branch points of $\pi:Y\to X$ is equidistant with respect to exactly one of these three coordinates. Setting $y=\sqrt[3]{2} y_i$ for $i=1,2,3$, we obtain as equation for the three connected components of $\Yb$:
\[
\Yb_1:\; y_1^3=x_1^4+x_1^2, \qquad \Yb_i:\; y_i^3=x_i^2+x_i, \quad \text{ for }i=2,3.
\]   
 We refer to \cite[\S5.1.2]{MichelDiss} for more details on the computation.

We note that the coordinates $x_i, y_i$ for $i=1,2,3$ may all be already defined over a totally ramified extension of $\QQ_2.$ Therefore we may choose a $2$-Frobenius element $\phi$ that fixes all these coordinates. (This follows as in Remark \ref{rem:good_red}.) 

Since $|\Yb_i(\FF_2)|=3$, we find that
\[
P(\phi^{-1}, T)=(1+2T^2)^3.
\]
As in Example \ref{exa:Picard2}, we let $\chi$ be the  $1$-dimensional unramified representation that sends $\phi$ to $1/\sqrt{-2}$. Then $\rho_0:=\rho\otimes\chi^{-1}$ is an Artin representation, and it factors  through the group $I\rtimes \langle \overline{\phi}\rangle\simeq C_2^2\times S_3.$ This group has label (24, 14) in the library of small groups. A character table can be found on \cite{GroupNames}. We write $\sigma\in I$ for one of the elements of order $3$. As in Example \ref{exa:Picard2}, we have that
\[
\sigma^\ast(x_i, y_i)=(x_i,\zeta_3y_i)
\]
for an element $\zeta_3\in \FF_4$ of order $3$. We conclude $\phi\sigma\phi^{-1}=\sigma^2\in I$. Since $g(\Yb_i/\langle \sigma\rangle)=0$ for all $i$, it follows as in Example \ref{exa:Picard2} that $\rho_0$ is the sum of three irreducible representations of dimension $2$.  To determine which of the four $2$-dimensional irreducible representations occur, it suffices to consider the restriction of $\rho_0$ to the Sylow $2$-subgroup of $I$. 

Let $\tau_1\in I$ be induced by $\tau_1(\zeta)=\zeta^{-1}$. Then $\tau_1$ acts as $\tau_1(x_1, y_1)=(x_1+1, y_1)$ on $\Yb_1$ and permutes the  other two components of $\Yb$. The automorphism $\tau_2\in I$ induced by $\tau_2(\zeta)=\zeta^5$ acts trivially on $\Yb_1$ and as an automorphism of order $2$ on $\Yb_2$ and $\Yb_3$. Both  statements follow by considering the specialization of the branch points of $\pi$  to the projective lines with coordinates $x_i$, see \cite[\S5.1.2]{MichelDiss}. We conclude that $\dim_\CC V^{\langle \tau_1, \tau_2\rangle}=0$.

We conclude that $\rho_0$ is the sum of the three  irreducible $2$-dimensional representations $V_i$ of $\overline{G}$ with 
\[
\Tr(\rho_0|_{V_i}(\tau_j))=(-1)^{\delta_{i,j}}2.
\]
The irreducible $2$-dimensional representation of $\overline{G}$ that does not occur satisfies $\Tr(\rho_0(\tau_i))=2$ for all $i$. It is the representation labeled $\rho_9$ in the character table on \cite{GroupNames}.
\end{exa}

\vspace{5ex}\noindent {\small Irene Bouw, Duc Khoi Do, Stefan Wewers\\ Institut
	f\"ur Algebra und Zahlentheorie\\ Universit\"at
	Ulm\\  {\tt irene.bouw@uni-ulm.de,
		stefan.wewers@uni-ulm.de}}

\end{document}